\documentclass{amsart}

\usepackage{amssymb,bm,amsfonts, stmaryrd, amsmath, amsthm, enumerate,,colonequals,xifthen,tikz-cd,hyperref}

\usepackage[capitalise]{cleveref}

\crefformat{equation}{(#2#1#3)}
\crefrangeformat{equation}{(#3#1#4--#5#2#6)}
\crefmultiformat{equation}{(#2#1#3}{ \& #2#1#3)}{, #2#1#3}{ \& #2#1#3)}

\crefformat{enumi}{(#2#1#3)}
\crefrangeformat{enumi}{(#3#1#4--#5#2#6)}
\crefmultiformat{enumi}{(#2#1#3}{ \& #2#1#3)}{, #2#1#3}{ \& #2#1#3)}

\newtheorem{theorem}[subsection]{Theorem}
\newtheorem{proposition}[subsection]{Proposition}
\newtheorem{lemma}[subsection]{Lemma}
\newtheorem{corollary}[subsection]{Corollary}

\theoremstyle{definition}

\newtheorem{discussion}[subsection]{}

\theoremstyle{remark}
\newtheorem*{ack}{Acknowledgements}
\newtheorem{remark}[subsection]{Remark}

\numberwithin{equation}{subsection}

\theoremstyle{plain}

\newcounter{intro}
\newtheorem{introthm}[intro]{Theorem}

\tikzcdset{
  column sep/wide/.initial=+6em,
  column sep/Wide/.initial=+7.2em,
  column sep/wider/.initial=+8.4em,
  column sep/Wider/.initial=+9.6em,
  column sep/broad/.initial=+10.8em,
  column sep/Broad/.initial=+12em,
  column sep/broader/.initial=+13.2em,
  column sep/Broader/.initial=+14.4em,
  row sep/wide/.initial=+4.5em,
  row sep/Wide/.initial=+5.4em,
  row sep/wider/.initial=+6.3em,
  row sep/Wider/.initial=+7.2em,
}

\newcommand{\nospacepunct}[1]{\makebox[0pt][l]{\,#1}}

\tikzcdset{ampersand replacement=\&,nodes in empty cells}

\renewcommand{\epsilon}{\varepsilon}
\renewcommand{\phi}{\varphi}
\renewcommand{\theta}{\vartheta}

\newcommand{\ai}{\an{\infty}}
\newcommand{\an}[1]{\mathrm{A}_{#1}}
\newcommand{\augCoa}{\cat{C}\mathrm{oa}^{aug}}
\newcommand{\cat}[1]{{\mathsf{#1}}}
\newcommand{\cder}{\operatorname{cder}}
\newcommand{\cmor}{\operatorname{cmor}}
\newcommand{\Coder}[2][]{\operatorname{Cder}^{#1}(#2)}
\DeclareMathOperator{\coker}{coker}
\DeclareMathOperator{\cone}{cone}
\renewcommand{\hom}[3][]{\operatorname{Hom}^{#1}(#2,#3)}
\newcommand{\Hom}[4][]{\operatorname{Hom}^{#1}_{#2}(#3,#4)}
\DeclareMathOperator{\id}{id}
\newcommand{\obs}[2]{\operatorname{obs}^{#1}(#2)}
\newcommand{\obsr}[2]{\overline{\operatorname{obs}}^{#1}(#2)}

\newcommand{\shift}{{\scriptstyle\mathsf{\Sigma}}}
\newcommand{\susp}{\Sigma}
\newcommand{\tcoa}[2][]{{\rm T}_{\ifthenelse{\isempty{#1}}{}{(#1)}}^{c}(#2)}

\newcommand{\BN}{\mathbb{N}}

\newcommand{\BZ}{\mathbb{Z}}

\newcommand{\bmalpha}{\bm{\alpha}}
\newcommand{\bmbeta}{\bm{\beta}}

\newcommand{\bmphi}{\bm{\varphi}}
\newcommand{\bmpsi}{\bm{\psi}}

\newcommand{\bmf}{\bm{f}}
\newcommand{\bmg}{\bm{g}}

\title{Transfer of \texorpdfstring{$A_\infty$-}{A-infinity }structures to projective resolutions}

\date{\today}



\author[J.~C.~Letz]{Janina C. Letz}
\address{Janina~C.~Letz,
Faculty of Mathematics,
Bielefeld University,
PO Box 100 131,
33501 Bielefeld,
Germany \newline
UCLA Department of Mathematics,
PO Box 951555, 
Los Angeles, CA 90095, 
United States}
\email{jletz@math.uni-bielefeld.de}

\begin{document}

\begin{abstract}
Let $B$ be a (unital) algebra, or an $\ai$-algebra, over a commutative ring $Q$. We show that the algebra structure on $B$ lifts to a (unital) $\ai$-algebra structure on any projective resolution of $B$. We show a similar result for module structures. The proofs are constructive and provide a step-by-step construction of the $\ai$-structures. 
\end{abstract}

\maketitle


\section{Introduction}

Algebras that are associative up to strong homotopy can be found in various areas. Such algebras were first introduced by Stasheff \cite{Stasheff:1963a,Stasheff:1963b} and are called $\ai$-algebras; for an introduction see \cite{Keller:2001}. An $\ai$-algebra is a graded module $A$ equipped with higher multiplication maps $\mu_n \colon A^{\otimes n} \to A$ satisfying some relations. One can think of $\mu_1$ as a differential, $\mu_2$ as a multiplication, and $\mu_3$ the homotopy up to which $\mu_2$ is associative.

In homological algebra $\ai$-algebras turn up in several places. Working over a field, one can endow the homology of a differential graded (dg) algebra with an $\ai$-algebra structure. This $\ai$-algebra is quasi-isomorphic to the original dg algebra. On the other hand, working over a commutative ring any projective resolution of a dg algebra admits an $\ai$-algebra structure so that the dg algebra and its resolution are quasi-isomorphic as algebras. These examples are instances of the same phenomena, the \emph{transfer} of an algebra structure along a quasi-isomorphism:

\begin{introthm} \label{mainthm}
Let $Q$ be a commutative ring and $B$ an $\ai$-algebra over $Q$. Given a quasi-isomorphism $\epsilon_1 \colon A \to B$ of complexes with $A$ a bounded below complex of projective modules, there exists an $\ai$-algebra structure on $A$ and a quasi-isomorphism $\epsilon \colon A \to B$ of $\ai$-algebras extending $\epsilon_1$. 

Furthermore, if $B$ is strictly unital with unit $1_B$ and $A$ splits as $A = Q \oplus \bar{A}$ with $\epsilon_1(1_A) = 1_B$ and $\partial(1_A) = 0$, where $1_A$ is a free generator of the summand $Q$ of $A$, then there exists a strictly unital $\ai$-algebra structure on $A$ with the unit $1_A$. 

Moreover, the $\ai$-algebra structure is unique up to $\ai$-homotopy.
\end{introthm}

This result can be found in \cref{unique_transfer_ai_alg,unique_transfer_unital_ai_alg}. It is similar to the results in \cite{Petersen:2020} which uses the language of operads. The aim of this paper is to provide a complete and constructive proof using algebraic language. 

Many previous results focused on the case when $Q$ a field and thus $\epsilon_1$ a homotopy equivalence, or even a strong deformation retract. There are two techniques typically used, the obstruction method and the tensor trick; for an overview see \cite[Section~2]{Johansson/Lambe:2001}. The first result was shown by Kadeishvili \cite{Kadeishvili:1980}; further results appear in \cite{Gugenheim:1982,Gugenheim/Stasheff:1986,Gugenheim/Lambe:1989,Gugenheim/Lambe/Stasheff:1991,Huebschmann/Kadeishvili:1991,Markl:2006,Proute:2011}. Explicit formulas are given in \cite[Theorem~3.4]{Merkulov:1999}, and for strictly unital $\ai$-algebra structures in \cite[Theorem~10]{Cheng/Getzler:2008}.

Burke \cite{Burke:2018} was the first to obtain a transfer result along a quasi-isomorphism of complexes over a commutative ring. He focused on the case when the quasi-isomorphism $\epsilon_1$ is surjective. This simplifies many computations, as the lift of a complex map along a surjective quasi-isomorphism yields a commutative diagram. In contrast, the lift along a quasi-isomorphism yields a diagram that is only commutative up to homotopy. Similar results over commutative rings appear in \cite{Burke:notes} and \cite[Section~5]{Briggs/Cameron/Letz/Pollitz:2025}. 

In this work we use a similar method as \cite{Kadeishvili:1980,Burke:2018}. The strategy is to inductively construct the higher multiplications. We analyze the obstruction of $(\mu_1, \ldots, \mu_n, 0, 0, \ldots)$ being an $\ai$-algebra structure; see \cref{sec:obstruction}. While it is doable to carry out the computations for the obstruction of an $\ai$-algebra structure by hand, the computations for morphisms of $\ai$-algebras and homotopies between them are significantly more cumbersome. To simplify the computations we use the bar construction, which allows us to work with coderivations on the tensor coalgebra. Our proof of the transfer theorem is constructive; it yields a step-by-step instruction on how to construct the $\ai$-algebra structure. 

When working with $\ai$-structures one has to find a balance between writing  out explicit formulas, and avoiding the more technical expressions and computations. While we skip some technical computations in this paper, we write out many formulas explicitly for the convenience of the reader. 

We finish the article by sketching transfer results for $\ai$-module structures in \cref{sec:modules}. The arguments are similar as for $\ai$-algebras, and we skip the proofs.

\begin{ack} 
This article is inspired by the work on \cite{Briggs/Cameron/Letz/Pollitz:2025}, and I want to thank Benjamin Briggs, James Cameron and Josh Pollitz for helpful discussions. The author was partly supported by the Deutsche Forschungsgemeinschaft (SFB-TRR 358/1 2023 - 491392403) and by the Alexander von Humboldt Foundation in the framework of a Feodor Lynen research fellowship endowed by the German Federal Ministry of Education and Research. 
\end{ack}

\section{The tensor coalgebra}

Throughout $Q$ is a fixed commutative ring. Any module, complex or algebra is a $Q$-module, $Q$-complex or $Q$-algebra, respectively, and any map is $Q$-linear. Unadorned tensor products and Hom sets are taken over $Q$. 

For a graded object $M$, we use lower grading in the sense that $M=\{M_i\}_{i\in \BZ}$; unless otherwise specified, graded will mean $\BZ$-graded. The degree of an element $m$ in $M$, denoted $|m|$, is the integer $i$ for which $m\in M_i$. The differential of a complex $M$ is denoted $\partial^M$. 

\begin{discussion}
We let $\susp$ denote the suspension functor: For a graded object $M$, we let $\susp M$ be the graded object with $(\susp M)_i=M_{i-1}$ for all $i\in \BZ$, and there is a naturally defined degree 1 map $\shift \colon M \to \susp M$ given by $m \mapsto \shift m$ where the latter is regarded as an element of $(\susp M)_{|m|+1}=M_{|m|}$. If $M$ is a complex, then we equip $\susp M$ with the differential $\partial^{\susp M} \colonequals -\partial^M$ which makes $\susp M$ a complex. 
\end{discussion}

\begin{discussion}
Let $M,N$ be graded modules. We let $\hom{M}{N}$ denote the graded module with 
\begin{equation*}
\hom[i]{M}{N} \colonequals \prod_{j\in \BZ} \hom{M_j}{N_{i+j}} \quad \text{for all } i \in \BZ\,.
\end{equation*}
When $M,N$ are complexes, we equip $\hom{M}{N}$ with a differential
\begin{equation*}
\partial^{\hom{M}{N}}(f) \colonequals \partial^N f - (-1)^{|f|} f \partial^M
\end{equation*}
making the graded module a complex. A map $f \colon M \to N$ is \emph{null-homotopic}, if it is a boundary; that is there exists $g \colon M \to N$ such that
\begin{equation*}
f = \partial^{\hom{M}{N}}(g) = \partial^N g - (-1)^{|g|} g \partial^M\,.
\end{equation*}
In this case, we say $f$ is null-homotopic with \emph{homotopy} $g$. 
\end{discussion}

\begin{discussion}
Let $M,N$ be graded modules. We let $M \otimes N$ denote the graded module with 
\begin{equation*}
(M \otimes N)_i \colonequals \coprod_{j \in \BZ} M_j \otimes N_{i-j} \quad \text{for all } i \in \BZ\,.
\end{equation*}
For maps of graded modules $f \colon M \to M'$ and $g\colon N \to N'$, we adopt the Koszul sign-rule:
\begin{equation*}
(f\otimes g) (m\otimes n) \colonequals (-1)^{|g||m|}f(m)\otimes g(n)\quad \text{for all } m \in M, n \in N\,.
\end{equation*}
When $M,N$ are complexes, we equip $M \otimes N$ with a differential
\begin{equation*}
\partial^{M \otimes N} \colonequals \partial^M \otimes \id_N + \id_M \otimes \partial^N
\end{equation*}
making the graded module a complex. 
\end{discussion}

\begin{discussion}
For a tuple of positive integers $\bmalpha = (\alpha_1, \ldots, \alpha_p) \in \BN^p$ we write $|\bmalpha| = \alpha_1 + \ldots + \alpha_p$. Given a family of morphisms $f_k \colon M^{\otimes k} \to M$ we set
\begin{equation*}
\bmf^{\otimes \bmalpha} \colonequals f_{\alpha_1} \otimes \cdots \otimes f_{\alpha_p}\,.
\end{equation*}
We use the convention, that $\BN^0 = \{()\}$ is a set with one element and $\bmf^{\otimes ()} \colonequals \id_Q$. In particular $\bmf^{\otimes ()} \otimes g \cong g$ for any map $g$.
\end{discussion}

\begin{discussion}
Let $M$ be a (graded) module. The \emph{tensor coalgebra $\tcoa{M}$ of $M$} has the underlying (bi)graded module 
\begin{equation*}
\tcoa{M} \colonequals \coprod_{k \geqslant 0} \tcoa[k]{M} \quad \text{with} \quad \tcoa[k]{M} \colonequals M^{\otimes k}
\end{equation*}
and the comultiplication
\begin{equation*}
\Delta(m_1 \otimes \ldots \otimes m_k) \colonequals \sum_{i=0}^k (m_1 \otimes \ldots \otimes m_i) \otimes (m_{i+1} \otimes \ldots \otimes m_k)\,.
\end{equation*}
The counit is the natural projection onto tensor degree 0. When $M$ is a graded module, then we denote the grading that $\tcoa{M}$ inherits from $M$ by $\tcoa{M}_\bullet$.

A coalgebra $C$ is \emph{coaugmented}, if there is a morphism of degree zero of coalgebras $\eta \colon Q \to C$, such that $\eta$ is a splitting of the counit as $Q$-modules. The tensor coalgebra is coaugmented with the natural inclusion $\eta \colon Q \to \tcoa{M}$.

A morphism of coaugmented coalgebras $F \colon C \to C'$ is a morphism of coalgebras such that $F \eta = \eta'$. We denote the category of coaugmented graded coalgebras by $\augCoa$. 
\end{discussion}

\begin{discussion}
The morphisms of coaugmented tensor coalgebras are completely determined by the restriction $\tcoa[>0]{M} \to N$. Explicitly, the natural map
\begin{equation*}
\begin{gathered}
\Hom{\augCoa}{\tcoa{M}}{\tcoa{N}} \to \hom{\tcoa[>0]{M}}{N} \,, \\
\text{given by} \quad F \mapsto (\tcoa[>0]{M} \to \tcoa{M} \xrightarrow{F} \tcoa{N} \to \tcoa[1]{N} = N)
\end{gathered}
\end{equation*}
is a bijection; see for example \cite[Lemme~1.1.2.2a]{LefevreHasegawa:2003}. Its inverse is
\begin{equation*}
\begin{gathered}
\cmor \colon \prod_{k \geqslant 1} \hom{M^{\otimes k}}{N} \to \Hom{\augCoa}{\tcoa{M}}{\tcoa{N}} \,, \\
\text{given by} \quad (f_k)_{k \geqslant 1} \mapsto \id_Q + \sum_{k \geqslant 1} \sum_{p=1}^k \sum_{\substack{\bmalpha \in \BN^p\\|\bmalpha|=k}} \bmf^{\otimes \bmalpha}
\end{gathered}
\end{equation*}
when pre-composed with the isomorphism coming from the universal property of the coproduct. The map $\cmor$ is well-defined, since the sum over $(p,\bmalpha)$ is finite for each $k$. We abuse notation and also denote by $\cmor$ the pre-composition of $\cmor$ with the inclusion of $\coprod_{k=1}^n \hom{M^{\otimes k}}{N}$. 
\end{discussion}

\begin{discussion}
Let $F,G \colon C \to C'$ be morphisms of graded coalgebras. A \emph{$(F,G)$-coderivation of degree $d$} is a $Q$-linear map $D \colon C \to C'$ of degree $d$ such that the following diagram commutes
\begin{equation*}
\begin{tikzcd}
C \ar[r,"{\Delta}"] \ar[d,"{D}" swap] \& C \otimes C \ar[d,"{D \otimes G + F \otimes D}"] \\
C' \ar[r,"{\Delta}"] \& C' \otimes C' \nospacepunct{.}
\end{tikzcd}
\end{equation*}
We denote the set of $(F,G)$-coderivations of degree $d$ by $\Coder[d]{F,G}$. When $F = G = \id_C$, we write $\Coder[d]{C} \colonequals \Coder[d]{\id_C,\id_C}$. 

If $D$ is a coderivation of odd degree $d$, then $D^2$ is a coderivation of degree $2d$. The same need not hold for even degree.
\end{discussion}

\begin{discussion}
On the tensor coalgebra a coderivation is completely determined by the restriction $\tcoa{M} \to N$. Explicitly, for morphisms of bigraded coalgebras $F,G \colon \tcoa{M} \to \tcoa{N}$ the natural map
\begin{equation} \label{coder_hom}
\begin{gathered}
\Coder[d]{F,G} \to \hom[d]{\tcoa{M}_\bullet}{N} \\
\text{given by} \quad D \mapsto (\tcoa{M} \xrightarrow{D} \tcoa{N} \to \tcoa[1]{N} = N)
\end{gathered}
\end{equation}
is a bijection; see for example \cite[Lemme~1.1.2.2b]{LefevreHasegawa:2003}. The degree of the domain and the codomain refers to the grading that $\tcoa{M}$ and $\tcoa{N}$ inherit from the grading of $M$ and $N$, respectively. With $F = \cmor((f_k)_{k \geqslant 1})$ and $G = \cmor((g_k)_{k \geqslant 1})$, its inverse is given by
\begin{equation*}
\begin{gathered}
\cder_{(F,G)} \colon \prod_{k \geqslant 0} \hom[d]{M^{\otimes k}}{N} \to \Coder[d]{F,G} \\
\text{given by} \quad (b_k)_{k \geqslant 0} \mapsto \sum_{k \geqslant 0} \sum_{\substack{u+v+w=k\\u,v,w \geqslant 0}} \sum_{p=0}^u \sum_{\substack{\bmalpha \in \BN^p \\ |\bmalpha| = u}} \sum_{q=0}^w \sum_{\substack{\bmbeta \in \BN^q \\ |\bmbeta| = w}} (\bmf^{\otimes \bmalpha} \otimes b_v \otimes \bmg^{\otimes \bmbeta})
\end{gathered}
\end{equation*}
when pre-composed with the isomorphism coming from the universal property of the coproduct. The map is well-defined since the sum over $(u,v,w,p,\bmalpha,q,\bmbeta)$ is finite for each $k$. For convenience we write $\cder \colonequals \cder_{(\id,\id)}$. We abuse notation and also denote by $\cder_{(F,G)}$ the pre-composition of $\cder_{(F,G)}$ with the inclusion of $\coprod_{k=1}^n \hom{M^{\otimes k}}{M}$ for any $n$. 
\end{discussion}

\section{\texorpdfstring{$A_\infty$}{A-infinity}-algebras} \label{sec:aialg}

In this section we give the definition of $\ai$-algebras, and their morphisms and homotopy equivalences. These structures are closely connected to structures on a tensor coalgebra. $\ai$-algebras were introduced by Stasheff \cite{Stasheff:1963a,Stasheff:1963b}; also see \cite{Keller:2001}. 

\begin{discussion}
An \emph{$\ai$-algebra} is a graded module $A = \{A_d\}_{d \in \BZ}$ equipped with maps
\begin{equation*}
\mu^A_k \colon A^{\otimes k} \to A \quad \text{for } k \geq 1
\end{equation*}
of degree $(k-2)$ satisfying the \emph{Stasheff identities}
\begin{equation} \label{alg_stasheff}
\sum_{\substack{u+v+w = k\\u,w \geqslant 0,v \geqslant 1}} (-1)^{u+vw} \mu^A_{u+1+w} (\id_A^{\otimes u} \otimes \mu^A_v \otimes \id_A^{\otimes w}) = 0 \quad \text{for } k \geq 1\,.
\end{equation}
The Stasheff identity is a relation on maps $A^{\otimes k} \to A$. 

For $k=1$, the Stasheff identity is $\mu^A_1 \mu^A_1 = 0$. As $\mu^A_1$ is a map of degree $-1$; this means $A$ can be viewed as a complex with differential $\mu^A_1$. When we want to emphasize the complex structure, we write $\partial^A$ for $\mu^A_1$.
\end{discussion}

At first glance the sign in \cref{alg_stasheff} seems strange. It is in fact chosen so that the identities on $\susp A$ are much simpler.

\begin{discussion}
Let $A$ be an $\ai$-algebra. We define $m^A_k \colon (\susp A)^{\otimes k} \to \susp A$ via $m^A_k \shift^{\otimes k} = -\shift \mu^A_k$. It is straightforward to check, that the $k$th Stasheff identity \cref{alg_stasheff} is equivalent to
\begin{equation*} \label{alg_shift_stasheff}
\sum_{\substack{u+v+w = k\\u,w \geqslant 0,v \geqslant 1}} m^A_{u+1+w} (\id_{\susp A}^{\otimes u} \otimes m^A_v \otimes \id_{\susp A}^{\otimes w}) = 0\,.
\end{equation*}
The shifted multiplication maps $m^A_k$ are all of degree $-1$. Hence an $\ai$-algebra structure $(\mu^A_k)_{k \geqslant 1}$ on a graded module $A$ is equivalent to a coderivation $D \colon \tcoa{\susp A} \to \tcoa{\susp A}$ of degree $-1$ with $D \eta = 0$ and $D^2 = 0$. In fact, the coderivation is given by $D = \cder((m^A_k)_{k \geqslant 1})$. 

What is more, for any $\ai$-algebra $A$ the data $(\tcoa{\susp A},\Delta,D)$ is a dg coalgebra. This is called the \emph{(non-unital) bar construction}. 
\end{discussion}

There are different sign conventions for \cref{alg_stasheff} and the conversion to the shifted multiplication maps; see \cref{sign_convention} for an in-depth discussion.

\begin{discussion}
Let $A$ and $B$ be $\ai$-algebras. A \emph{morphism of $\ai$-algebras} $\phi \colon A \to B$ consists of maps
\begin{equation*}
\phi_k \colon A^{\otimes k} \to B \quad \text{for } k \geq 1
\end{equation*}
of degree $(k-1)$ satisfying
\begin{equation} \label{alg_mor_stasheff}
\sum_{\substack{u+v+w=k \\ u,w \geqslant 0, v \geqslant 1}} (-1)^{u+vw} \phi_{u+1+w} (\id_A^{\otimes u} \otimes \mu^A_v \otimes \id_A^{\otimes w}) = \sum_{p=1}^k \sum_{\substack{\bmalpha \in \BN^p \\ |\bmalpha|=k}} (-1)^{t_1(\bmalpha)} \mu^B_p \bmphi^{\otimes \bmalpha}
\end{equation}
where $t_1(\bmalpha) = \sum_{\ell=1}^p (p-\ell)(\alpha_\ell-1)$. A morphism $\phi \colon A \to B$ is \emph{strict}, if $\phi_k = 0$ for $k \geq 2$. 

For $k=1$, the identity \cref{alg_mor_stasheff} is $\phi_1 \mu^A_1 = \mu^B_1 \phi_1$. As $(A,\mu^A_1)$ and $(B,\mu^B_1)$ are complexes, we can view $\phi_1 \colon A \to B$ as a chain map. We say a morphism of $\ai$-algebras $\phi \colon A \to B$ is a \emph{quasi-isomorphism of $\ai$-algebras}, if $\phi_1$ is a quasi-isomorphism of complexes.
\end{discussion}

\begin{discussion}
We can interpret these morphisms in terms of the tensor coalgebra: We define $f_k \colon (\susp A)^{\otimes k} \to \susp B$ via $f_k \shift^{\otimes k} = \shift \phi_k$. It is straightforward to check, that \cref{alg_mor_stasheff} is equivalent to
\begin{equation*}
\sum_{\substack{u+v+w=k \\ u,w \geqslant 0, v \geqslant 1}} f_{u+1+w} (\id_A^{\otimes u} \otimes m^A_v \otimes \id_A^{\otimes w}) = \sum_{p=1}^k \sum_{\substack{\bmalpha \in \BN^p \\ |\bmalpha|=k}} m^B_p \bmf^{\otimes \bmalpha}\,.
\end{equation*}
The shifted maps $f_k$ are of degree zero. A morphism of $\ai$-algebras $(\phi_k)_{k \geqslant 1}$ is equivalent to a morphism of coaugmented coalgebras $F \colon \tcoa{\susp A} \to \tcoa{\susp B}$ such that $F D^A = D^B F$ where $D^A$ and $D^B$ the coderivations induced by the $\ai$-algebra structures on $A$ and $B$, respectively. In fact, the morphism of coaugmented coalgebras is given by $F = \cmor((f_k)_{k \geqslant 1})$. 
\end{discussion}

\begin{discussion} \label{alg_mor_composition}
Let $\phi \colon A \to B$ and $\psi \colon B \to C$ be morphisms of $\ai$-algebras. Their composition is given by
\begin{equation*}
(\psi \circ \phi)_k \colonequals \sum_{p=1}^k \sum_{\substack{\bmalpha \in \BN^p\\|\bmalpha|=k}} (-1)^{t_1(\bmalpha)} \psi_p \bmphi^{\otimes \bmalpha} \quad \text{for } k \geq 1\,.
\end{equation*}
It is straightforward to check that this is a morphism of $\ai$-algebras. On the tensor coalgebras this corresponds to the usual composition $\cmor((g_k)_{k \geqslant 1}) \circ \cmor((f_k)_{k \geqslant 1}) = \cmor((h_k)_{k \geqslant 1})$ where $f_k$, $g_k$ and $h_k$ are the shifted morphisms corresponding to $\phi_k$, $\psi_k$ and $(\psi \phi)_k$, respectively.

The identity morphism $\id_A$ of an $\ai$-algebra $A$ is the strict morphism of $\ai$-algebra with $(\id_A)_1$ the identity on the module $A$.
\end{discussion}

\begin{discussion}
Let $\phi, \psi \colon A \to B$ be morphisms of $\ai$-algebras. We say $\phi$ and $\psi$ are \emph{homotopy equivalent}, and write $\phi \simeq \psi$, if there are maps
\begin{equation*}
\sigma_k \colon A^{\otimes k} \to B \quad \text{for } k \geqslant 1
\end{equation*}
of degree $k$ satisfying
\begin{equation} \label{alg_htpy_stasheff}
\begin{aligned}
\phi_k - & \psi_k = \sum_{\substack{u+v+w=k\\u,w \geqslant 0, v \geqslant 1}} (-1)^{u+vw} \sigma_{u+1+w} (\id_A^{\otimes u} \otimes \mu^A_v \otimes \id_A^{\otimes w}) \\
& + \sum_{\substack{u+v+w=k\\u,w \geqslant 0, v \geqslant 1}} \sum_{p=0}^v \sum_{\substack{\bmalpha \in \BN^p \\ |\bmalpha| = v}} \sum_{q=0}^w \sum_{\substack{\bmbeta \in \BN^q \\ |\bmbeta| = w}} (-1)^{u+vq+t_2(\bmalpha,\bmbeta)} \mu^B_{p+1+q} (\bmphi^{\otimes \bmalpha} \otimes \sigma_v \otimes \bmpsi^{\otimes \bmbeta})\,,
\end{aligned}
\end{equation}
where $t_2(\bmalpha,\bmbeta) = t_1(\bmalpha) + t_1(\bmbeta) + (u-p)(q+1)$. We call $\sigma$ an \emph{$\ai$-homotopy} between $\phi$ and $\psi$.

A morphism of $\ai$-algebras $\phi \colon A \to B$ is a \emph{homotopy equivalence} of $\ai$-algebras, if there exists a morphism of $\ai$-algebras $\psi \colon B \to A$ such that $\psi \phi \simeq \id_A$ and $\phi \psi \simeq \id_B$. 
\end{discussion}

\begin{discussion}
On the tensor coalgebras $\ai$-homotopies correspond to derivations: We define $s_k \colon (\susp A)^{\otimes k} \to \susp B$ via $s_k \shift^{\otimes k} = - \shift \sigma_k$. It is straightforward to check that \cref{alg_htpy_stasheff} is equivalent to
\begin{equation*} \label{alg_htpy_shift_stasheff}
\begin{aligned}
f_k - & g_k = \sum_{\substack{u+v+w=k\\u,w \geqslant 0, v \geqslant 1}} s_{u+1+w} (\id_A^{\otimes u} \otimes m^A_v \otimes \id_A^{\otimes w}) \\
& +\sum_{\substack{u+v+w=k\\u,w \geqslant 0, v \geqslant 1}} \sum_{p=0}^u \sum_{\substack{\bmalpha \in \BN^p \\ |\bmalpha| = u}} \sum_{q=0}^w \sum_{\substack{\bmbeta \in \BN^q \\ |\bmbeta| = w}} m^B_{p+1+q} (\bmf^{\otimes \bmalpha} \otimes s_v \otimes \bmg^{\otimes \bmbeta})\,,
\end{aligned}
\end{equation*}
where $f_k$ and $g_k$ are the shifted maps corresponding to $\phi_k$ and $\psi_k$, respectively. The shifted maps $s_k$ are all of degree $1$. Hence an $\ai$-homotopy $(\sigma_k)_{k \geqslant 1}$ is equivalent to a $(F,G)$-coderivation $S \colon \tcoa{\susp A} \to \tcoa{\susp B}$ of degree $1$ with $F-G = D^B S + S D^A$, where $F$ and $G$ are the coalgebra morphisms corresponding to $\phi$ and $\psi$, respectively, and $D^A$ and $D^B$ the coderivations induced by the $\ai$-algebra structures on $A$ and $B$, respectively. In fact, the homotopy is given by $S = \cder_{(F,G)}((s_k)_{k \geqslant 1})$. 
\end{discussion}

\section{Obstructions} \label{sec:obstruction}

In this section we show the technical results used in the proof of the transfer theorem. For this we analyze the terms of the $k$th Stasheff identity, that do not contain $\mu^A_k$. Some of the results are mentioned in \cite[Section~B.1]{LefevreHasegawa:2003}.

\begin{discussion}
For a positive integer $n$, a graded module $A$ equipped with multiplication maps $\mu^A_1, \ldots, \mu^A_n$ satisfying the Stasheff identities \cref{alg_stasheff} for $1 \leq k \leq n$, is called an \emph{$\an{n}$-algebra}. Similar as above, an $\an{n}$-algebra structure $(\mu^A_k)_{k \leqslant n}$ on $A$ induces a coderivation $D = \cder(m^A_1, \ldots, m^A_n)$, where $m^A_1, \ldots, m^A_n$ are the shifted multiplication maps. 

For an $\an{n}$-algebra, we set
\begin{equation*}
\obs{n}{A} \colonequals \sum_{\substack{u+v+w = n+1\\u,w \geqslant 0,n \geqslant v \geqslant 2}} (-1)^{u+vw} \mu^A_{u+1+w} (\id_A^{\otimes u} \otimes \mu^A_v \otimes \id_A^{\otimes w})\,.
\end{equation*}
This is a map $A^{\otimes (n+1)} \to A$ of degree $n-2$. Rewriting the Stasheff identity \cref{alg_stasheff} for $n+1$ yields: A map $\mu^A_{n+1} \colon A^{\otimes (n+1)} \to A$ makes $A$ an $\an{n+1}$-algebra if and only if
\begin{equation} \label{obs_alg_stasheff}
\obs{n}{A} = -\partial(\mu^A_{n+1})\,;
\end{equation}
that is $-\obs{n}{A}$ is null-homotopic as a map of complexes with homotopy $\mu^A_{n+1}$. 
\end{discussion}

The following identity is clear for an $\an{n+1}$-algebra, but it even holds for an $\an{n}$-algebra.

\begin{lemma} \label{obs_alg_del}
For any $\an{n}$-algebra $A$ one has
\begin{equation*}
\partial(\obs{n}{A}) = 0\,.
\end{equation*}
\end{lemma}
\begin{proof}
We set $D \colonequals \cder(m^A_1, \ldots, m^A_n)$ and $D_1 \colonequals \cder(m^A_1)$. Then $D_1^2 = 0$ and, since $A$ is an $\an{n}$-algebra, the restriction of $D^2$ to $(\susp A)^{\otimes (k+i)} \to (\susp A)^{\otimes i}$ is zero for $k < n$ and $i \geq 1$. By definition, the restriction of $\bar{D} \colonequals D - D_1$ to $(\susp A)^{\otimes i} \to (\susp A)^{\otimes (i+j)}$ is zero for $i \geq 1$ and $j \geq 0$. Hence
\begin{equation*}
D^2 \bar{D} = 0 = \bar{D} D^2
\end{equation*}
when restricted to $(\susp A)^{\otimes (n+1)} \to \susp A$. By straightforward computations we have
\begin{equation*}
D_1 \bar{D}^2 = \bar{D}^2 D_1
\end{equation*}
when restricted to $(\susp A)^{\otimes (n+1)} \to \susp A$. After replacing the shifted multiplications $m^A_k$ by the multiplications $\mu^A_k$, we obtain
\begin{equation*}
-\partial^A \obs{n}{A} = (-1)^{n+1} \obs{n}{A} \partial^{A^{\otimes (n+1)}}\,.
\end{equation*}
We used that 
\begin{equation*}
(\id_{\susp A}^{\otimes u} \otimes m^A_1 \otimes \id_{\susp A}^{\otimes w}) \shift^{\otimes (n+1)} = (-1)^{n+1} \shift^{\otimes (n+1)} (\id_A^{\otimes u} \otimes \mu^A_1 \otimes \id_A^{\otimes w})\,.
\end{equation*}
for $u+v = n$. As $|\obs{n}{A}| = n-2$ we showed the desired identity.
\end{proof}

\begin{discussion}
A morphism $\phi \colon A \to B$ of $\an{n}$-algebras consists of morphisms $\phi_1, \ldots, \phi_n$ satisfying \cref{alg_mor_stasheff} for $1 \leq k \leq n$. Similarly as above, a morphism of $\an{n}$-algebras $(\phi_k)_{k \leqslant n}$ induces a morphism of coaugmented coalgebras $\cmor(f_1, \ldots, f_n)$, where $f_1, \ldots, f_n$ are the shifted morphisms. 

For a morphism of $\an{n}$-algebras $\phi \colon A \to B$ we set
\begin{equation*}
\begin{aligned}
\obs{n}{\phi} \colonequals \sum_{p=2}^n & \sum_{\substack{\bmalpha \in \BN^p \\ |\bmalpha|=n+1}} (-1)^{t_1(\bmalpha)} \mu^B_p \bmphi^{\otimes \bmalpha} \\
&- \sum_{\substack{u+v+w=n+1 \\ u,w \geqslant 0, n \geqslant v \geqslant 2}} (-1)^{u+vw} \phi_{u+1+w} (\id_A^{\otimes u} \otimes \mu^A_v \otimes \id_A^{\otimes w})\,.
\end{aligned}
\end{equation*}
This is a map $A^{\otimes (n+1)} \to B$ of degree $n-1$. Rewriting the identity \cref{alg_mor_stasheff} for $n+1$ yields: If $A$ and $B$ are $\an{n+1}$-algebras, then a map $\phi_{n+1} \colon A^{\otimes (n+1)} \to B$ makes $\phi$ a morphism of $\an{n+1}$-algebras, if and only if
\begin{equation} \label{obs_alg_mor_stasheff}
\obs{n}{\phi} = \partial(\phi_{n+1}) + \mu^B_{n+1} \phi_1^{\otimes (n+1)} - \phi_1 \mu^A_{n+1}\,.
\end{equation}
\end{discussion}

\begin{lemma} \label{obs_mor_del}
For any morphism of $\an{n}$-algebras $\phi \colon A \to B$ one has
\begin{equation*}
\partial(\obs{n}{\phi}) = \phi_1 \obs{n}{A} - \obs{n}{B} \phi_1^{\otimes (n+1)}\,.
\end{equation*}
\end{lemma}
\begin{proof}
Let $D^A$, $\bar{D}^A$, $D^B$ and $\bar{D}^B$ be the coderivations induced from the multiplications on $A$ and $B$, respectively, as in \cref{obs_alg_del}. We set
\begin{equation*}
\begin{gathered}
F \colonequals \cmor(f_1, \ldots, f_n) \,,\quad F_1 \colonequals \cmor(f_1) \quad \text{and} \quad \bar{F} \colonequals F - F_1\,.
\end{gathered}
\end{equation*}
By the same argument as in \cref{obs_alg_del} we have
\begin{equation*}
\bar{D}^B F D^A = \bar{D}^B D^B F \,,\quad F D^A \bar{D}^A = D^B F \bar{D}^A \quad \text{and} \quad \bar{F} D^A D^A = 0 = D^B D^B \bar{F}
\end{equation*}
when restricted to $(\susp A)^{\otimes (n+1)} \to \susp B$. By straightforward computations we have
\begin{equation*}
D^B_1 (\bar{F} \bar{D}^A - \bar{D}^B \bar{F}) + (\bar{F} \bar{D}^A - \bar{D}^B \bar{F})D^A_1 = F_1 \bar{D}^A \bar{D}^A - \bar{D}^B \bar{D}^B F_1
\end{equation*}
when restricted to $(\susp A)^{\otimes (n+1)} \to \susp B$. Replacing the shifted maps by the maps of the $\ai$-structures, we obtain
\begin{equation*}
\partial^B \obs{n}{\phi} + (-1)^n \obs{n}{\phi} \partial^{A^{\otimes (n+1)}} = \phi_1 \obs{n}{A} - \obs{n}{B} \phi_1\,.
\end{equation*}
This is the desired identity.
\end{proof}

\begin{lemma} \label{obs_comp}
For morphisms $\phi \colon A \to B$ and $\psi \colon B \to C$ of $\an{n}$-algebras one has
\begin{equation*}
\obs{n}{\psi \circ \phi} = \psi_1 \obs{n}{\phi} + \obs{n}{\psi} \phi_1^{\otimes (n+1)} + \partial(\sum_{p=2}^n \sum_{\substack{\bmalpha \in \BN^p\\|\bmalpha|=n+1}} (-1)^{t_1(\bmalpha)} \psi_p \bmphi^{\otimes \bmalpha})\,,
\end{equation*}
where $\psi \circ \phi$ is the composition of morphisms of $\ai$-algebras defined in \cref{alg_mor_composition}. 
\end{lemma}

Morally, the last summand consists of the terms of $(\psi \circ \phi)_{n+1}$, which do not contain $\phi_{n+1}$ or $\psi_{n+1}$. 

\begin{proof}
We use the same notation and arguments as in \cref{obs_alg_del,obs_mor_del}, and further define $G$, $G_1$ and $\bar{G}$ using $\psi$ analogously as $F$, $F_1$ and $\bar{F}$ are defined using $\phi$. By straightforward computations we have
\begin{equation*}
\begin{aligned}
(GF - G_1 F_1) & \bar{D}^A - \bar{D}^C (GF - G_1 F_1) \\
&= G_1 (\bar{F}\bar{D}^A -\bar{D}^B\bar{F}) + (\bar{G}\bar{D}^B-\bar{D}^C\bar{G}) F_1 + D^C_1\bar{G}\bar{F} - \bar{G}\bar{F} D^A_1
\end{aligned}
\end{equation*}
when restricted to $(\susp A)^{\otimes (n+1)} \to \susp C$. Replacing the shifted maps by the maps of the $\ai$-structures, we obtain the desired identity.
\end{proof}

\begin{discussion}
An $\an{n}$-homotopy $\sigma$ of morphisms of $\an{n}$-algebras $\phi, \psi \colon A \to B$ consists of maps $\sigma_1, \ldots, \sigma_n$ satisfying \cref{alg_htpy_stasheff} for $1 \leq k \leq n$. Similarly as above, an $\an{n}$-homotopy $(\sigma_k)_{k \leqslant n}$ induces a $(F,G)$-coderivation $S = \cder_{(F,G)}(s_1, \ldots, s_n)$ of degree 1, where $F = \cmor(f_1, \ldots, f_n)$ and $G = \cmor(g_1, \ldots, g_n)$, and $s_1, \ldots, s_n$ are the shifted maps.

For an $\an{n}$-homotopy $\sigma$ of $\phi, \psi \colon A \to B$ we set
\begin{equation*}
\begin{aligned}
& \obs{n}{\sigma} \colonequals \sum_{\substack{u+v+w=n+1\\u,w \geqslant 0, n \geqslant v \geqslant 2}} (-1)^{u+vw} \sigma_{u+1+w} (\id_A^{\otimes u} \otimes \mu^A_v \otimes \id_A^{\otimes w}) \\
& + \sum_{\substack{u+v+w=n+1\\u,w \geqslant 0, n \geqslant v \geqslant 1}} \sum_{p=0}^u \sum_{\substack{\bmalpha \in \BN^p \\ |\bmalpha| = u}} \sum_{\substack{q=0 \\ p+q < n}}^w \sum_{\substack{\bmbeta \in \BN^q \\ |\bmbeta| = w}} (-1)^{u+vq+t_2(\bmalpha,\bmbeta)} \mu^B_{p+1+q} (\bmphi^{\otimes \bmalpha} \otimes \sigma_v \otimes \psi^{\otimes \bmbeta})\,.
\end{aligned}
\end{equation*}
This is a map $A^{\otimes (n+1)} \to B$ of degree $n$. Rewriting the identity \cref{alg_htpy_stasheff} for $n+1$ yields: If $\phi, \psi$ are morphisms of $\an{n+1}$-algebras, then a map $\sigma_{n+1}$ makes $\sigma$ an $\an{n+1}$-homotopy if and only if
\begin{equation} \label{obs_alg_htpy_stasheff}
\begin{aligned}
\phi_{n+1} - \psi_{n+1} = \obs{n}{\sigma} & + \partial(\sigma_{n+1}) + \sigma_1 \mu^A_{n+1} \\
&+ (-1)^n \mu^B_{n+1} \sum_{\substack{u+w=n\\u,w \geqslant 0}} (\phi_1^{\otimes u} \otimes \sigma_1 \otimes \psi_1^{\otimes w})\,.
\end{aligned}
\end{equation}
\end{discussion}

\begin{lemma} \label{obs_htpy_del}
For any $\an{n}$-homotopy $\sigma$ of morphisms of $\an{n}$-algebras $\phi, \psi \colon A \to B$ one has
\begin{equation*}
\begin{aligned}
\obs{n}{\phi} - \obs{n}{\psi} =  \partial(&\obs{n}{\sigma}) + \sigma_1 \obs{n}{A} \\
&- (-1)^n \obs{n}{B} \sum_{\substack{u+w=n\\u,w \geqslant 0}} (\phi_1^{\otimes u} \otimes \sigma_1 \otimes \psi_1^{\otimes w})\,.
\end{aligned}
\end{equation*}
\end{lemma}
\begin{proof}
We use the same notation and arguments as in \cref{obs_alg_del,obs_mor_del,obs_comp}. We set
\begin{equation*}
\begin{gathered}
S \colonequals \cder_{(F,G)}(s_1, \ldots, s_n) \,,\quad S_1 \colonequals \cder_{(F_1,G_1)}(s_1) \quad \text{and} \quad \bar{S} \colonequals S - S_1\,.
\end{gathered}
\end{equation*}
For these the following identities hold
\begin{equation*}
\begin{gathered}
(F - G) \bar{D}^A = D^B S \bar{D}^A + S D^A \bar{D}^A \,, \quad \bar{D}^B (F - G) = \bar{D}^B D^B S + \bar{D}^B S D^A \\
\text{and} \quad D^B D^B \bar{S} = 0 = \bar{S} D^A D^A
\end{gathered}
\end{equation*}
when restricted to $(\susp A)^{\otimes (n+1)} \to \susp B$. By straightforward computations we have
\begin{equation*}
\begin{aligned}
(\bar{F} \bar{D}^A - \bar{D}^B & \bar{F}) - (\bar{G} \bar{D}^A - \bar{D}^B \bar{G}) \\
&= D^B_1 (\bar{D}^B \bar{S} + \bar{S} \bar{D}^A) - (\bar{D}^B \bar{S} + \bar{S} \bar{D}^A) D^A_1 + S_1 \bar{D}^A \bar{D}^A - \bar{D}^B \bar{D}^B S_1
\end{aligned}
\end{equation*}
when restricted to $(\susp A)^{\otimes (n+1)} \to \susp B$. Replacing the shifted maps by the maps of the $\ai$-structures, we obtain the desired identity.
\end{proof}

\section{Transfer of \texorpdfstring{$A_\infty$}{A-infinity}-structures} \label{sec:transfer}

We now give the transfer and lifting results for non-unital $\ai$-algebras. The arguments are inspired by \cite{Kadeishvili:1980,Burke:2018}, who proved similar results over a field or for surjective quasi-isomorphisms, respectively. To work with non-surjective quasi-isomorphisms we first need the suitable lifting results for chain complexes.

\begin{discussion} \label{lifting_cx_surj}
We consider a commutative diagram of chain maps
\begin{equation*}
\begin{tikzcd}
U \ar[r,"f"] \ar[d,"g" swap] \& X \ar[d,"{\sim}" swap,"g'"] \\
V \ar[r,"f'"] \& Y
\end{tikzcd}
\end{equation*}
where $g$ is degreewise injective with $\coker(g)$ a bounded below complex of projectives, and $g'$ a quasi-isomorphism that is degreewise surjective. Then there exists a lift $h \colon V \to X$ such that each triangle commutes; see for example \cite[Section~7]{Dwyer/Spalinski:1995}.
\end{discussion}

If one wants to remove the condition that $g'$ is degreewise surjective, this comes at the price that the lower right triangle need not commute honestly, but only up to homotopy. This result is well-known, through it seems to not appear in the literature in the setting of chain complexes. We give a proof for completeness.

\begin{lemma} \label{lifting_cx}
We consider a commutative diagram of chain maps
\begin{equation*}
\begin{tikzcd}
U \ar[r,"f"] \ar[d,"g" swap] \& X \ar[d,"{\sim}" swap,"g'"] \\
V \ar[r,"f'"] \& Y
\end{tikzcd}
\end{equation*}
where $g$ is degreewise injective with $\coker(g)$ a bounded below complex of projectives, and $g'$ a quasi-isomorphism. Then there exists a chain map $h \colon V \to X$ such that $hg = f$ and $f' \simeq g'h$. Moreover, if $s$ is the homotopy witnessing $f' \simeq g'h$, then $sg = 0$. 
\end{lemma}
\begin{proof}
We denote by $\operatorname{cocyl}(g') \colonequals \cone(\susp^{-1} Y \to \susp^{-1} \cone(g'))$ the mapping cocylinder of $g'$. This is the mapping cylinder in the opposite category; for the mapping cylinder see for example \cite[1.5.5]{Weibel:1994}. The underlying graded module of the mapping cocylinder is $\susp^{-1} Y \oplus X \oplus Y$. This yields a commutative diagram
\begin{equation*}
\begin{tikzcd}
U \ar[d,"="] \ar[r,"f"] \& X \ar[d,"{\left(\begin{smallmatrix} 0 \\ \id_X \\ g' \end{smallmatrix}\right)}","\sim" swap] \\
U \ar[r] \ar[d,"g" swap] \& \operatorname{cocyl}(g') \ar[d,"{\sim}" swap,"{\left(\begin{smallmatrix} 0 & 0 & \id_Y \end{smallmatrix}\right)}"] \\
V \ar[r,"f'"] \& Y \nospacepunct{.}
\end{tikzcd}
\end{equation*}
It is straightforward to check that the vertical maps on the right are homotopy equivalences and thus quasi-isomorphisms. The bottom square satisfies the assumptions of \cref{lifting_cx_surj}, and hence there exists a chain map $(h_1, h_2, h_3)^T \colon V \to \operatorname{cocyl}(g')$ such that each triangle commutes. It is straightforward to check that $h_2$ is a chain map and $f = h_2 g$ and $f'-g' h_2 = \partial(h_1)$ and $h_1 g = 0$. Hence $h = h_2$ and $s = h_1$ satisfy the desired properties.
\end{proof}

\begin{lemma} \label{null_homotopic_qi}
Let $g \colon X \to Y$ be a quasi-isomorphism and $f \colon F \to X$ a chain map where $F$ a bounded below complex of projectives. If $gf$ is null-homotopic with homotopy $s$, then there exists a map $h \colon \susp F \to X$ such that $f = \partial(h)$ and $s \simeq gh$. 
\end{lemma}
\begin{proof}
We have the following commutative diagram
\begin{equation*}
\begin{tikzcd}
F \ar[r,"f"] \ar[d] \& X \ar[d,"\sim" swap,"g"] \\
\cone(\id_F) \ar[r,"{(gf,s)}"] \& Y \nospacepunct{.}
\end{tikzcd}
\end{equation*}
By \cref{lifting_cx}, there exists a chain map $(h_1,h_2) \colon \cone(\id_F) \to X$ such that the upper triangle commutes and the lower triangle commutes up to homotopy. It is straightforward to check that $h = h_2$ satisfies the desired conditions.
\end{proof}

\begin{theorem} \label{transfer_ai_alg}
Let $B$ be an $\ai$-algebra and let $\epsilon_1 \colon A \to B$ be a quasi-isomorphism of complexes with $(A,\mu^A_1)$ a bounded complex of projectives. Then there exists an $\ai$-algebra structure on $A$ and a quasi-isomorphism of $\ai$-algebras $\epsilon \colon A \to B$ extending $\epsilon_1$. If $\epsilon_1$ is degreewise surjective, then the morphism of $\ai$-algebras $\epsilon$ can be chosen to be strict.
\end{theorem}
\begin{proof}
We inductively construct maps $\mu^A_n$ and $\epsilon_n$ satisfying the $n$th Stasheff identities \cref{alg_stasheff,alg_mor_stasheff}. 

For $n = 1$ we set $\mu^A_1 \colonequals \partial^A$. It is clear, that the first Stasheff identities are satisfied for $\mu^A_1$ and $\epsilon_1$.

Let $n \geq 2$ and assume $A$ is an $\an{n}$-algebra and $\epsilon$ a morphism of $\an{n}$-algebras. By \cref{obs_alg_del}, the obstruction map $\obs{n}{A}$ is a chain map. By \cref{obs_mor_del} we have
\begin{equation*}
\epsilon_1 \obs{n}{A} = \partial(\obs{n}{\epsilon}) + \obs{n}{B} \epsilon_1^{\otimes (n+1)} = \partial(\obs{n}{\epsilon} - \mu^B_{n+1} \epsilon_1^{\otimes (n+1)})\,;
\end{equation*}
the latter equality uses the $(n+1)$st Stasheff identity \cref{obs_alg_stasheff} for $B$ and that $\epsilon_1$ is a chain map. Using \cref{null_homotopic_qi} we obtain a map $-\mu^A_{n+1} \colon A^{\otimes (n+1)} \to A$ of degree $n-1$ and a map $\epsilon_{n+1} \colon A^{\otimes (n+1)} \to B$ of degree $n$ such that 
\begin{equation*}
\obs{n}{A} = - \partial(\mu^A_{n+1}) \quad \text{and} \quad \obs{n}{\epsilon} - \mu^B_{n+1} \epsilon_1^{\otimes (n+1)} + \epsilon_1 \mu^A_{n+1} = \partial(\epsilon_{n+1})\,.
\end{equation*}
These are precisely \cref{obs_alg_stasheff,obs_alg_mor_stasheff}, and hence $A$ is an $\an{n+1}$-algebra and $\epsilon$ is a morphism of $\an{n+1}$-algebras. 
\end{proof}

\begin{theorem} \label{lift_ai_mor}
Let $\epsilon \colon B \to C$ be a quasi-isomorphism of $\ai$-algebras and $\psi \colon A \to C$ a morphism of $\ai$-algebras with $(A,\mu^A_1)$ a bounded complex of projectives. Then there exists a morphism of $\ai$-algebras $\phi \colon A \to B$ such that $\epsilon \phi \simeq \psi$ as morphisms of $\ai$-algebras. If $\epsilon$ is degreewise surjective, then $\phi$ can be chosen such that $\epsilon \phi = \psi$. 
\end{theorem}
\begin{proof}
We inductively construct maps $\phi_n$ and $\sigma_n$ such that $\phi \colon A \to B$ is a morphism of $\an{n}$-algebras and $\sigma$ is an $\an{n}$-homotopy witnessing $\epsilon \phi \simeq \psi$. 

For $n=1$ we get $\phi_1$ and $\sigma_1$ satisfying 
\begin{equation*}
\mu^B_1 \phi_1 = \phi_1 \mu^A_1 \quad \text{and} \quad \epsilon_1 \phi_1 - \psi_1 = \mu^C_1 \sigma_1 + \sigma_1 \mu^A_1
\end{equation*}
from \cref{lifting_cx}. These are the first Stasheff identities \cref{alg_mor_stasheff,alg_htpy_stasheff}.

Let $n \geq 2$ and assume $\phi$ is a morphism of $\an{n}$-algebras and $\sigma$ is an $\an{n}$-homotopy witnessing $\epsilon \phi \simeq \psi$. By \cref{obs_comp,obs_htpy_del}, we obtain
\begin{equation*}
\begin{aligned}
&\quad \epsilon_1(\obs{n}{\phi} - \mu^B_{n+1} \phi_1^{\otimes (n+1)} + \phi_1 \mu^A_{n+1}) \\
&= \partial(\obs{n}{\sigma}) + \obs{n}{\psi} + \sigma_1 \obs{n}{A} \\
&\,\, - (-1)^n \obs{n}{C} \sum_{\substack{u+w=n\\u,w \geqslant 0}} (\phi_1^{\otimes u} \otimes \sigma_1 \otimes \psi^{\otimes w}) - \obs{n}{\epsilon} \phi_1^{\otimes (n+1)} \\
&\,\,- \partial(\sum_{p=2}^n \sum_{\substack{\bmalpha \in \BN^p\\|\bmalpha|=n+1}} (-1)^{t_1(\bmalpha)} \epsilon_p \bmphi^{\otimes \bmalpha}) - \epsilon_1 \mu^B_{n+1} \phi_1^{\otimes (n+1)} + \epsilon_1 \phi_1 \mu^A_{n+1}\,.
\end{aligned}
\end{equation*}
Since $A$ and $C$ are $\ai$-algebras and $\epsilon$ and $\psi$ are morphisms of $\ai$-algebras, the Stasheff identities \cref{obs_alg_stasheff,obs_alg_mor_stasheff} yield
\begin{equation*}
\begin{aligned}
&\quad \epsilon_1(\obs{n}{\phi} - \mu^B_{n+1} \phi_1^{\otimes (n+1)} + \phi_1 \mu^A_{n+1}) \\
&= \partial(\obs{n}{\sigma} + \psi_{n+1} - \sum_{p=2}^{n+1} \sum_{\substack{\bmalpha \in \BN^p\\|\bmalpha|=n+1}} (-1)^{t_1(\bmalpha)} \epsilon_p \bmphi^{\otimes \bmalpha}) \\
&\quad + (\epsilon_1 \phi_1 - \psi_1) \mu^A_{n+1} - \sigma_1 \partial(\mu^A_{n+1}) \\
&\quad - \mu^C_{n+1} ((\epsilon_1 \phi_1)^{\otimes (n+1)} - \psi^{\otimes (n+1)}) + (-1)^n \partial(\mu^C_{n+1}) \sum_{\substack{u+w=n\\u,w \geqslant 0}} (\phi_1^{\otimes u} \otimes \sigma_1 \otimes \psi^{\otimes w})\,.
\end{aligned}
\end{equation*}
By the base case we have $\epsilon_1 \phi_1 - \psi_1 = \partial(\sigma_1)$, and thus
\begin{equation*}
(\epsilon_1 \phi_1 - \psi_1) \mu^A_{n+1} - \sigma_1 \partial(\mu^A_{n+1}) = \partial(\sigma_1) \mu^A_{n+1} - \sigma_1 \partial(\mu^A_{n+1}) = \partial(\sigma_1 \mu^A_{n+1})\,.
\end{equation*}
Similarly, we obtain for the terms involving $\mu^C_{n+1}$:
\begin{equation*}
\begin{aligned}
&\quad - \mu^C_{n+1} ((\epsilon_1 \phi_1)^{\otimes (n+1)} - \psi^{\otimes (n+1)}) + (-1)^n \partial(\mu^C_{n+1}) \sum_{\substack{u+w=n\\u,w \geqslant 0}} (\phi_1^{\otimes u} \otimes \sigma_1 \otimes \psi^{\otimes w}) \\
&= (-1)^n \partial(\mu^C_{n+1} \sum_{\substack{u+w=n\\u,w \geqslant 0}} ((\epsilon_1 \phi_1)^{\otimes u} \otimes \sigma_1 \otimes \psi_1^{\otimes w}))\,.
\end{aligned}
\end{equation*}
So we have shown that $\epsilon_1(\obs{n}{\phi} - \mu^B_{n+1} \phi_1^{\otimes (n+1)} + \phi_1 \mu^A_{n+1})$ is null-homotopic with homotopy
\begin{equation*}
\begin{aligned}
\obs{n}{\sigma} & + \psi_{n+1} + \sigma_1 \mu^A_{n+1} \\
& - \sum_{p=2}^{n+1} \sum_{\substack{\bmalpha \in \BN^p\\|\bmalpha|=n+1}} (-1)^{t_1(\bmalpha)} \epsilon_p \bmphi^{\otimes \bmalpha} + (-1)^n \mu^C_{n+1} \sum_{\substack{u+w=n\\u,w \geqslant 0}} ((\epsilon_1 \phi_1)^{\otimes u} \otimes \sigma_1 \otimes \psi_1^{\otimes w})\,.
\end{aligned}
\end{equation*}
By \cref{null_homotopic_qi}, there exists a map $\phi_{n+1} \colon A^{\otimes (n+1)} \to B$ of degree $n$ and a map $-\sigma_{n+1} \colon A^{\otimes (n+1)} \to C$ of degree $n+1$ such that
\begin{equation*}
\begin{gathered}
\obs{n}{\phi} - \mu^B_{n+1} \phi_1^{\otimes (n+1)} + \phi_1 \mu^A_{n+1} = \partial(\phi_{n+1}) \quad \text{and} \\
\begin{aligned}
\obs{n}{\sigma} + \psi_{n+1} + & \sigma_1 \mu^A_{n+1} - \sum_{p=1}^{n+1} \sum_{\substack{\bmalpha \in \BN^p\\|\bmalpha|=n+1}} (-1)^{t_1(\bmalpha)} \epsilon_p \bmphi^{\otimes \bmalpha} \\
&+ (-1)^n \mu^C_{n+1} \sum_{\substack{u+w=n\\u,w \geqslant 0}} ((\epsilon_1 \phi_1)^{\otimes u} \otimes \sigma_1 \otimes \psi_1^{\otimes w}) = -\partial(\sigma_{n+1})\,.
\end{aligned}
\end{gathered}
\end{equation*}
These identities are precisely \cref{obs_alg_mor_stasheff,obs_alg_htpy_stasheff}, and hence $\phi$ is a morphism of $\an{n+1}$-algebras and $\sigma$ an $\an{n+1}$-homotopy witnessing $\epsilon \phi \simeq \psi$. 
\end{proof}

\begin{corollary} \label{unique_transfer_ai_alg}
The transfer of the $\ai$-algebra structure from $B$ to $A$ in \cref{transfer_ai_alg} is unique up to homotopy equivalence.
\end{corollary}
\begin{proof}
Let $A$ and $B$ be as in \cref{transfer_ai_alg}, and let $(A,\hat{\mu}_k)$ and $(A,\tilde{\mu}_k)$ be $\ai$-algebra structures on $A$ with quasi-isomorphisms of $\ai$-algebras $\hat{\epsilon}_k,\tilde{\epsilon} \colon A \to B$. By \cref{lift_ai_mor}, there exists a morphism of $\ai$-algebras $\phi \colon (A,\hat{\mu}_k) \to (A,\tilde{\mu}_k)$ such that $\tilde{\epsilon} \phi \simeq \hat{\epsilon}$ up to an $\ai$-homotopy $\sigma$. In particular, the morphism of $\ai$-algebras $\phi$ is a quasi-isomorphism. Using \cref{lift_ai_mor} again there exists a lift $\psi \colon (A,\hat{\mu}) \to (A,\tilde{\mu})$ of $\id_A$ along $\phi$, with $\phi \psi \simeq \id_A$. Analogously we have $\psi' \colon (A,\hat{\mu}) \to (A,\tilde{\mu})$ with $\psi' \phi \simeq \id_A$. Now it is straightforward to check that $\psi \simeq \psi'$, and hence $\phi$ is a homotopy equivalence.
\end{proof}

\section{Transfer of strictly unital \texorpdfstring{$A_\infty$}{A-infinity}-structures} \label{sec:transfer_unital}

In many situations the algebra $B$ is unital. This naturally leads to the question whether this property is preserved by the transfer theorem. This is true under some mild conditions on the free resolution. Such results have been shown by \cite[Theorem~10]{Cheng/Getzler:2008} over a field, and over a commutative ring for surjective quasi-isomorphisms by \cite{Burke:2018,Briggs/Cameron/Letz/Pollitz:2025}. 

\begin{discussion}
An $\ai$-algebra structure on $A$ is \emph{strictly unital}, if there exists a map of degree zero $\eta^A \colon Q \to A$, such that
\begin{equation} \label{alg_unital}
\begin{gathered}
\mu^A_2 (\id_A \otimes \eta^A) = \id_A = \mu^A_2 (\eta^A \otimes \id_A) \quad \text{and} \\
\mu^A_k (\id_A^{\otimes u} \otimes \eta^A \otimes \id_A^{\otimes w}) = 0 \quad \text{for } k \neq 2 \text{ and } u+w = k-1\,.
\end{gathered}
\end{equation}
The element $1_A \colonequals \eta^A(1_Q)$ is called the \emph{strict unit} of $A$. An $\ai$-algebra is \emph{split unital}, if it is strictly unital and $\eta^A$ splits as a map of graded modules. Then we can write $A \cong Q \oplus \bar{A}$ as graded modules. Any split unital $\ai$-algebra structure on $Q \oplus \bar{A}$ is fully determined by the restrictions of $\mu_k$ to
\begin{equation*}
\bar{A}^{\otimes k} \to \bar{A} \quad \text{and} \quad \bar{A}^{\otimes k} \to Q\,.
\end{equation*}
\end{discussion}

\begin{discussion}
A morphism of strictly unital $\ai$-algebras is a morphism of $\ai$-algebras $\phi \colon A \to B$ such that
\begin{equation} \label{alg_mor_unital}
\begin{gathered}
\phi_1 \eta^A = \eta^B \quad \text{and} \\
\phi_k (\id_A^{\otimes u} \otimes \eta^A \otimes \id_A^w) = 0 \quad \text{for } k \neq 1 \text{ and } u+w = k-1\,.
\end{gathered}
\end{equation}
A homotopy of morphisms of strictly unital $\ai$-algebras is a homotopy $\sigma$ of morphisms of $\ai$-algebras such that
\begin{equation} \label{alg_htpy_unital}
\sigma_k (\id_A^{\otimes u} \otimes \eta^A \otimes \id_A^w) = 0 \quad \text{for } u+w = k-1\,.
\end{equation}
\end{discussion}

\begin{discussion}
We say an $\an{n}$-algebra is strictly unital if \cref{alg_unital} holds for $1 \leq k \leq n$, and we require the corresponding morphisms and homotopies to satisfy \cref{alg_mor_unital,alg_htpy_unital} for $1 \leq k \leq n$. 

The strictly unital property also has an implication for the obstructions. For a strictly unital $\ai$-algebra $A$ the unit map $\eta^A$ induces a chain map
\begin{equation*}
\eta^A_k \colon \sum_{\substack{u+w=k-1\\u,w \geqslant 0}} A^{\otimes u} \otimes Q \otimes A^{\otimes w} \to A^{\otimes k}
\end{equation*}
given by $\id_A^{\otimes u} \otimes \eta_A \otimes \id_A^{\otimes w}$. Using the map $\eta^A_k$, the strictly unital condition \cref{alg_unital} for $k \neq 2$ is equivalent to $\mu^A_k \eta^A_k = 0$. Similarly, \cref{alg_mor_unital} can be expressed as $\phi_1 \eta^A_1 = \eta^B$ for $k = 1$ and $\phi_k \eta^A_k = 0$ for $k \neq 1$, and \cref{alg_htpy_unital} can be expressed as $\sigma_k \eta^A_k = 0$. 

It is straightforward to check that 
\begin{equation*}
\obs{n}{A} \eta^A_{n+1} = 0 \,,\quad \obs{n}{\phi} \eta^A_{n+1} = 0 \quad \text{and} \quad \obs{n}{\sigma} \eta^A_{n+1} = 0
\end{equation*}
for any $\an{n}$-algebra $A$, morphism of $\an{n}$-algebras $\phi$ and $\an{n}$-homotopy $\sigma$.
\end{discussion}

\begin{proposition} \label{transfer_unital_ai_alg}
Let $B$ be a strictly unital $\ai$-algebra and let $\epsilon_1 \colon A \to B$ be a quasi-isomorphism of complexes with $A$ a bounded below complex of projectives. We further assume that $A = Q \oplus \bar{A}$ as graded modules such that $\partial(Q) = 0$ and $(Q \to A \to B) = \eta^B$. Then there exists a split unital $\ai$-algebra structure on $A$ and a quasi-isomorphism of strictly unital $\ai$-algebras $\epsilon \colon A \to B$ extending $\epsilon_1$. If $\epsilon$ is degreewise surjective, then the morphism of strictly unital $\ai$-algebras $\epsilon$ can be chosen to be strict.
\end{proposition}
\begin{proof}
This proof is similar to that of \cref{transfer_ai_alg}. The key difference is that we first restrict the maps to $\bar{A}^{\otimes (n+1)}$, before we use the argument of \cref{transfer_ai_alg}. 

We let $\eta^A \colon Q \to A$ be the natural inclusion. It will serve as the unit map. For every positive integer $n$ the map $\eta^A_n$ induced by $\eta^A$ is degreewise injective, and $\coker(\eta^A_n) = \bar{A}^{\otimes n}$ is a bounded complex of projectives.

For $n=1$ we set $\mu^A_1 \colonequals \partial^A$. Then $\mu^A_1$, $\epsilon_1$ and $\eta^A$ satisfy the Stasheff identities \cref{alg_stasheff,alg_mor_stasheff} as well as the strictly unital conditions \cref{alg_unital,alg_mor_unital}.

For $n = 2$ we consider the commutative diagram
\begin{equation*}
\begin{tikzcd}
A \otimes Q + Q \otimes A \ar[rr] \ar[d,"\eta^A_2"] \&\& A \ar[d,"\epsilon_1"] \\
A \otimes A \ar[r,"{\epsilon_1 \otimes \epsilon_1}"] \& R \otimes R \ar[r,"{\mu^R_2}"] \& R
\end{tikzcd}
\end{equation*}
By \cref{lifting_cx}, there exists a chain map $\mu^A_2 \colon A \otimes A \to A$ such that the upper left triangle commutes and the lower right triangle commutes up to a homotopy $\epsilon_2$. The former ensures the strictly unital conditions \cref{alg_unital,alg_mor_unital}, the latter the Stasheff identity \cref{alg_mor_stasheff}. The Stasheff identity for algebras \cref{alg_stasheff} holds since $\mu^A_2$ is a chain map.

Let $n \geq 3$ and assume $A$ is a strictly unital $\an{n}$-algebra and $\epsilon$ a morphism of strictly unital $\an{n}$-algebras. Since $\obs{n}{A} \eta^A_{n+1} = 0$ and $\obs{n}{\phi} \eta^A_{n+1} = 0$, the obstructions factor through $\bar{A}^{\otimes (n+1)}$. We denote by $\obsr{n}{A}$ and $\obsr{n}{\phi}$ these maps. As graded modules $\bar{A}^{\otimes (n+1)}$ is a direct summand of $A^{\otimes (n+1)}$ and the maps $\obsr{n}{A}$ and $\obsr{n}{\phi}$ are in fact the restrictions of $\obs{n}{A}$ and $\obs{n}{\phi}$ to $\bar{A}^{\otimes (n+1)}$. Hence they satisfy the analogous identities as $\obs{n}{A}$ and $\obs{n}{\phi}$ in the proof of \cref{transfer_ai_alg}, and by the same argument as in the proof of \cref{transfer_ai_alg} we obtain morphisms $\bar{\mu}^A_{n+1} \colon \bar{A}^{\otimes (n+1)} \to A$ and $\bar{\epsilon}_{n+1} \colon \bar{A}^{\otimes (n+1)} \to B$. Their pre-composition with $A^{\otimes (n+1)} \to \bar{A}^{\otimes (n+1)}$ yields the desired morphisms.
\end{proof}

\begin{proposition} \label{lift_unital_ai_mor}
Let $\epsilon \colon B \to C$ be a quasi-isomorphism of strictly unital $\ai$-algebras and $\psi \colon A \to C$ a morphism of strictly unital $\ai$-algebras with $A$ split unital and $(A,\mu^A_1)$ a bounded complex of projectives. Then there exists a morphism of strictly unital $\ai$-algebras $\phi \colon A \to B$ such that $\epsilon \phi \simeq \psi$ as morphisms of strictly unital $\ai$-algebras.
\end{proposition}
\begin{proof}
Let $A = Q \oplus \bar{A}$ where $Q$ is generated by the strict unit $1_A$.

For $n=1$ we apply \cref{lifting_cx} to
\begin{equation*}
\begin{tikzcd}
Q \ar[r,"{\eta^B}"] \ar[d,"\eta^A" swap] \& B \ar[d,"{\epsilon_1}","\sim" swap] \\
A \ar[r,"{\psi_1}"] \& C
\end{tikzcd}
\end{equation*}
to obtain $\phi_1$ and $\sigma_1$. 

Let $n \geq 2$ and assume $\phi$ is a morphism of strictly unital $\an{n}$-algebras and $\sigma$ an $\an{n}$-homotopy witnessing $\epsilon \phi \simeq \psi$. As in the proof of \cref{transfer_unital_ai_alg}, we restrict the maps to $\bar{A}^{\otimes (n+1)}$ before we apply the arguments of the proof of \cref{lift_ai_mor}. This yields maps with the desired properties.
\end{proof}

\begin{corollary} \label{unique_transfer_unital_ai_alg}
The transfer of strictly unital $\ai$-algebra structure from $B$ to $A$ in \cref{transfer_unital_ai_alg} is unique up to homotopy equivalence. \qed
\end{corollary}

\section{\texorpdfstring{$A_\infty$}{A-infinity}-modules} \label{sec:modules}

Finally, we sketch the arguments for the transfer of $\ai$-module structures. The approach is analogous to that for $\ai$-algebras, so we mimic \cref{sec:aialg,sec:obstruction,sec:transfer,sec:transfer_unital}. 

\begin{discussion}
Let $A$ be an $\ai$-algebra over $Q$. An \emph{$\ai$-modules over $A$} is a graded module $M = \{M_d\}_{d \in \BZ}$ equipped with maps\begin{equation*}
\mu^M_k \colon A^{\otimes (k-1)} \otimes M \to M \quad \text{for } k \geq 1
\end{equation*}
of degree $(k-2)$ satisfying
\begin{equation} \label{mod_stasheff}
\begin{gathered}
\sum_{\substack{u+v+w = k\\u \geqslant 0, v,w \geqslant 1}} (-1)^{u+vw} \mu^M_{u+1+w} (\id_A^{\otimes u} \otimes \mu^A_v \otimes \id_A^{\otimes (w-1)} \otimes \id_M) \\
+ \sum_{\substack{u+v = k \\ u \geqslant 0, v \geqslant 1}} (-1)^u \mu^M_{u+1} (\id_A^{\otimes u} \otimes \mu^M_v) = 0\,.
\end{gathered}
\end{equation}
These identities are similar to \cref{alg_stasheff}, and provide relations of maps $A^{\otimes (k-1)} \otimes M \to M$. For convenience we drop the superscript on the higher multiplication maps and the subscript on the identity maps, and combine the sum as in \cref{alg_stasheff}. 

Over a strictly unital $\ai$-algebra an $\ai$-module is \emph{strictly unital}, if
\begin{equation*}
\begin{gathered}
\mu^M_2(\eta \otimes \id_M) = \id_M \quad \text{and} \\
\mu^M_k(\id^{\otimes u} \otimes \eta^A \otimes \id^{\otimes (w+1)}) = 0 \quad \text{for } k \geq 2 \text{ and } u+w = k-2\,.
\end{gathered}
\end{equation*}

For an $\ai$-module we define $m^M_k \colon (\susp A)^{\otimes (k-1)} \otimes M \to M$ via $m^M_k (\shift^{\otimes (k-1)} \otimes \id_M) = (-1)^{k-1} \mu^M_k$. It is straightforward to check, that the Stasheff identity \cref{mod_stasheff} is equivalent to
\begin{equation*}
\sum_{\substack{u+v+w = k\\u,w \geqslant 0, v \geqslant 1}} m_{u+1+w} (\id^{\otimes u} \otimes m_v \otimes \id^{\otimes w}) = 0\,.
\end{equation*}
The shifted maps $m_k$ are of degree $-1$. Similar to $\ai$-algebras, an $\ai$-module structure on $M$ can be encoded as a coderivation $D_M$ on the comodule $\tcoa{\susp A} \otimes M$ over $\tcoa{\susp A}$ with $D_M^2 = 0$; see for example \cite[Lemme~2.1.2.1b]{LefevreHasegawa:2003}. 
\end{discussion}

\begin{discussion}
Let $M$ and $N$ be $\ai$-modules over an $\ai$-algebra $A$. A \emph{morphism of $\ai$-modules} $\rho \colon M \to N$ consists of maps
\begin{equation*}
\rho_k \colon A^{\otimes (k-1)} \otimes M \to N \quad \text{for } k \geq 1
\end{equation*}
of degree $(k-1)$ satisfying
\begin{equation} \label{mod_mor_stasheff}
\sum_{\substack{u+v+w=k\\u,w \geqslant 0, v \geqslant 1}} (-1)^{u+vw} \rho_{u+1+w} (\id^{\otimes u} \otimes \mu_v \otimes \id^{\otimes w}) = \sum_{\substack{u+v=k\\u \geqslant 0, v \geqslant 1}} \mu_{u+1} (\id^{\otimes u} \otimes \rho_v)\,.
\end{equation}

A morphism of strictly unital $\ai$-modules needs to additionally satisfy
\begin{equation*}
\rho_k(\id^{\otimes u} \otimes \eta \otimes \id^{\otimes (w+1)}) = 0 \quad \text{for } k \geq 2 \text{ and } u+w = k-2\,.
\end{equation*}

We define $r_k \colon (\susp A)^{\otimes (k-1)} \otimes M \to N$ via $r_k (\shift^{\otimes (k-1)} \otimes \id) = (-1)^{k-1} \rho_k$. It is straightforward to check, that \cref{mod_mor_stasheff} is equivalent to
\begin{equation*} \label{mod_mor_shift_stasheff}
\begin{gathered}
\sum_{\substack{u+v+w=k\\u,w \geqslant 0, v \geqslant 1}} r_{u+1+w} (\id^{\otimes u} \otimes m_v \otimes \id^{\otimes w}) = \sum_{\substack{u+v=k\\u \geqslant 0, v \geqslant 1}} \mu_{u+1} (\id^{\otimes u} \otimes r_v)\,.
\end{gathered}
\end{equation*}
The shifted maps are of degree zero, and a morphism of $\ai$-modules corresponds to a morphism of comodules $R \colon \tcoa{\susp A} \otimes M \to \tcoa{\susp A} \otimes N$ satisfying $R D^M = D^N R$; see \cite[Lemme~2.1.2.1a]{LefevreHasegawa:2003}. 
\end{discussion}

\begin{discussion}
Let $\rho \colon M \to N$ and $\pi \colon L \to M$ be morphisms of $\ai$-modules over an $\ai$-algebra $A$. The composition $\rho \circ \pi \colon L \to N$ is given by
\begin{equation*}
(\rho \circ \pi)_k \colonequals (-1)^{k-1} \sum_{\substack{u+v=k\\u \geqslant 0, v \geqslant 1}} \rho_{u+1} (\id^{\otimes u} \otimes \pi_v)\,.
\end{equation*}
The identity morphism $\id_M \colon M \to M$ of $\ai$-modules is given by
\begin{equation*}
(\id_M)_k = \begin{cases}
\id_M & k = 1 \\
0 & \text{else}\,.
\end{cases}
\end{equation*}
\end{discussion}

\begin{discussion}
Let $\phi \colon A \to B$ be a morphism of $\ai$-algebras and $M$ an $\ai$-module over $B$. Then there exists a natural $\ai$-module structure over $A$ on $M$ given by
\begin{equation*}
\mu^{\phi_* M}_1 \colonequals \mu^M_1 \quad \text{and} \quad \mu^{\phi_* M}_k \colonequals \sum_{p=2}^k \sum_{\substack{\bmalpha \in \BN^{p-1}\\|\bmalpha|=k-1}} \mu^M_p (\bmphi^{\otimes \bmalpha} \otimes \id_M)\,;
\end{equation*}
cf.\@ \cite[Section~7]{Lu/Palmieri/Wu/Zhang:2004}. We write $\phi_* M$ for the restricted $\ai$-module over $A$. 
\end{discussion}

\begin{discussion}
A morphism of $\ai$-modules $\rho \colon M \to N$ is \emph{null-homotopic}, if there exist maps
\begin{equation*}
\tau_k \colon A^{\otimes (k-1)} \otimes M \to N \quad \text{for } k \geq 1
\end{equation*}
of degree $k$ satisfying
\begin{equation} \label{mod_hpty_stasheff}
\begin{aligned}
\rho_k = \sum_{\substack{u+v+w=k\\u,w \geqslant 0, v \geqslant 1}} (-1)^{u+vw} & \tau_{u+1+w} (\id^{\otimes u} \otimes \mu_v \otimes \id^{\otimes w}) \\
&+ \sum_{\substack{u+v=k\\u \geqslant 0, v \geqslant 1}} (-1)^u \mu_{u+1} (\id^{\otimes u} \otimes \tau_v)\,.
\end{aligned}
\end{equation}
We call $\tau$ an \emph{$\ai$-homotopy}. We say morphisms $\rho, \pi \colon M \to N$ are \emph{homotopic} and write $\rho \simeq \pi$, if $\rho - \pi$ is null-homotopic.

When $\rho \colon M \to N$ is a morphism of strictly unital $\ai$-modules, we say that it is null-homotopic if the above condition holds and
\begin{equation*}
\tau_k(\id^{\otimes u} \otimes \eta \otimes \id^{\otimes (w+1)}) = 0 \quad \text{for } k \geq 2 \text{ and } u+w = k-2\,.
\end{equation*}

We define $t_k \colon (\susp A)^{\otimes (k-1)} \otimes M \to N$ via $t_k (\shift^{\otimes (k-1)} \otimes \id) = (-1)^{k-1} \tau_k$. It is straightforward to check that \cref{mod_hpty_stasheff} is equivalent to
\begin{equation*}
r_k = \sum_{\substack{u+v+w=k\\u,w \geqslant 0, v \geqslant 1}} t_{u+1+w} (\id^{\otimes u} \otimes m_v \otimes \id^{\otimes w}) + \sum_{\substack{u+v=k\\u \geqslant 0, v \geqslant 1}} m_{u+1} (\id^{\otimes u} \otimes t_v)\,.
\end{equation*}
\end{discussion}

Analogously to \cref{sec:obstruction} we define the $\an{n}$-structures for modules, morphisms of modules, and homotopies of morphisms of modules. So we can study the obstructions, and we obtain several relations that these obstructions satisfy.

\begin{discussion}
For an $\an{n}$-module we set
\begin{equation*}
\obs{n}{M} \colonequals \sum_{\substack{u+v+w=n+1\\u,w \geqslant 0, n \geqslant v \geqslant 2}} (-1)^{u+vw} \mu_{u+1+w} (\id^{\otimes u} \otimes \mu_v \otimes \id^{\otimes w})\,.
\end{equation*}
Then a map $\mu^M_{n+1} \colon A^{\otimes n} \otimes M \to M$ makes $M$ an $\an{n+1}$-module if and only if
\begin{equation} \label{obs_mod_stasheff}
\obs{n}{M} = - \partial(\mu^M_{n+1})\,.
\end{equation}

For a morphism of $\an{n}$-modules $\rho \colon M \to N$ we set
\begin{equation*}
\begin{aligned}
\obs{n}{\rho} \colonequals \sum_{\substack{u+v=n+1\\u \geqslant 0, v \geqslant 1}} & \mu_{u+1} (\id^{\otimes u} \otimes \rho_v) \\
&- \sum_{\substack{u+v+w=n+1\\u,w \geqslant 0, n \geqslant v \geqslant 2}} (-1)^{u+vw} \rho_{u+1+w} (\id^{\otimes u} \otimes \mu_v \otimes \id^{\otimes w})\,.
\end{aligned}
\end{equation*}
If $M$ and $N$ are $\an{n+1}$-modules, then a map $\rho_{n+1} \colon A^{\otimes n} \otimes M \to N$ makes $\rho$ a morphism of $\an{n+1}$-modules if and only if
\begin{equation} \label{obs_mod_mor_stasheff}
\obs{n}{\rho} = \partial(\rho_{n+1}) + \mu^N_{n+1} (\id_A^{\otimes n} \otimes \rho_1) - \rho_1 \mu^M_{n+1}\,.
\end{equation}

For an $\an{n}$-homotopy $\tau$ witnessing that $\rho \colon M \to N$ is null-homotopic, we set
\begin{equation*}
\begin{aligned}
\obs{n}{\tau} \colonequals \sum_{\substack{u+v+w=n+1\\u,w \geqslant 0, n \geqslant v \geqslant 2}} (-1)^{u+vw} & \tau_{u+1+w} (\id^{\otimes u} \otimes \mu_v \otimes \id^{\otimes w}) \\
&+ \sum_{\substack{u+v=n+1\\u \geqslant 0, n \geqslant v \geqslant 2}} (-1)^u \mu_{u+1} (\id^{\otimes u} \otimes \tau_v)
\end{aligned}
\end{equation*}
If $\rho$ is a morphism of $\an{n+1}$-modules, then a map $\tau_{n+1} \colon A^{\otimes n} \otimes M \to N$ makes $\tau$ a $\an{n+1}$-homotopy witnessing $\rho \simeq 0$ if and only if
\begin{equation} \label{obs_mod_hpty_stasheff}
\rho_{n+1} = \obs{n}{\tau} + \partial(\tau_{k+1}) + \tau_1 \mu^M_{n+1} + (-1)^n \mu^N_{n+1} (\id_A^{\otimes n} \otimes \tau_1)\,.
\end{equation}
\end{discussion}

\begin{lemma} \label{obs_mod_del}
For any $\an{n}$-module $M$ one has
\begin{equation*}
\partial(\obs{n}{M}) = 0\,. \pushQED{\qed} \qedhere \popQED
\end{equation*}
\end{lemma}

\begin{lemma} \label{obs_mod_mor}
For any morphism of $\an{n}$-modules $\rho \colon M \to N$ one has
\begin{equation*}
\rho_1 \obs{n}{M} - \obs{n}{N} (\id^{\otimes n} \otimes \rho_1) = \partial(\obs{n}{\rho})\,. \pushQED{\qed} \qedhere \popQED
\end{equation*}
\end{lemma}

\begin{lemma} \label{obs_mod_comp}
For morphisms $\rho \colon M \to N$ and $\phi \colon L \to M$ of $\an{n}$-modules one has
\begin{equation*}
\obs{n}{\rho \pi} = \rho_1 \obs{n}{\pi} + \obs{n}{\rho} (\id^{\otimes n} \otimes \pi_1) + \partial(\sum_{\substack{u+v=n+1\\u \geqslant 1, v \geqslant 2}} \rho_{u+1} (\id^{\otimes u} \otimes \pi_v))\,. \pushQED{\qed} \qedhere \popQED
\end{equation*}
\end{lemma}

\begin{lemma} \label{obs_mod_rest}
For a morphism of $\an{n}$-algebras $\phi \colon A \to B$ and an $\an{n}$-module $M$ over $B$ one has
\begin{equation*}
\obs{n}{M} (\phi_1^{\otimes n} \otimes \id_M) - \obs{n}{\phi_* M} = \partial(\sum_{p=2}^n \sum_{\substack{\bmalpha \in \BN^{p-1} \\ |\bmalpha| = n}} \mu_p (\bmphi^{\otimes \bmalpha} \otimes \id_M))\,. \pushQED{\qed} \qedhere \popQED
\end{equation*}
\end{lemma}

\begin{lemma} \label{obs_mod_hpty}
For an $\an{n}$-homotopy $\tau$ witnessing that a morphism of $\an{n}$-modules $\rho \colon M \to N$ is null-homotopic one has
\begin{equation*}
\partial(\obs{n}{\tau}) = \obs{n}{\rho} - \tau_1 \obs{n}{M} + (-1)^n \obs{n}{N} (\id_A^{\otimes n} \otimes \tau_1)\,. \pushQED{\qed} \qedhere \popQED
\end{equation*}
\end{lemma}

This concludes the analysis of the obstruction terms, and we are ready for the transfer and lifting results for $\ai$-module structures.

\begin{proposition} \label{transfer_ai_mod}
We assume $A$ is an $\ai$-algebra with $(A,\mu^A_1)$ a bounded below complex of projectives. Let $M$ be an $\ai$-module over $A$ and let $\epsilon_1 \colon G \to M$ be a quasi-isomorphism with $(G,\mu^G_1)$ a bounded below complex of projectives. Then there exists an $\ai$-module structure on $G$ and a quasi-isomorphism of $\ai$-modules $\epsilon \colon G \to M$ extending $\epsilon_1$. If $\epsilon$ is degreewise surjective, then the morphism of $\ai$-modules $\epsilon$ can be chosen to be strict. If $A$ is split unital and $M$ strictly unital, then the $\ai$-module structure can be chosen to be strictly unital, and $\epsilon$ to be to be a morphism of strictly unital $\ai$-modules.
\end{proposition}
\begin{proof}
The proof is the same as the ones for \cref{transfer_ai_alg,transfer_unital_ai_alg}, using the formulas \cref{obs_mod_del,obs_mod_mor} for the obstructions.
\end{proof}

\begin{proposition} \label{lift_ai_mod_mor}
We assume $A$ is an $\ai$-algebras with $(A,\mu^A_1)$ a bounded below complex of projectives. Let $\epsilon \colon M \to N$ be a quasi-isomorphism of $\ai$-modules and $\pi \colon G \to N$ a morphism of $\ai$-modules with $(G,\mu^G_1)$ a bounded below complex of projectives. Then there exists a morphism of $\ai$-modules $\rho \colon A \to M$ such that $\epsilon \rho \simeq \pi$ as morphisms of $\ai$-modules. If $\epsilon_1$ is degreewise surjective, then $\rho$ can be chosen such that $\epsilon \rho = \pi$. If $A$ is split unital and $\epsilon$ and $\pi$ are morphisms of strictly unital $\ai$-modules, then $\rho$ can be chosen to be a morphism of strictly unital $\ai$-modules, and the $\ai$-homotopy can be chosen to respect the strictly unital structure.
\end{proposition}
\begin{proof}
The proof is the same as the ones for \cref{lift_ai_mor,lift_unital_ai_mor}, using the formulas \cref{obs_mod_comp,obs_mod_hpty} for the obstructions.
\end{proof}

\begin{corollary}
The transfer of (strictly unital) $\ai$-module structures from $M$ to $G$ in \cref{transfer_ai_mod} is unique up to homotopy equivalence. \qed
\end{corollary}

Finally, we combine the transfer of $\ai$-algebra structures and of $\ai$-module structures.

\begin{corollary}
Let $M$ be an $\ai$-module over an $\ai$-algebra $B$ and $\phi_1 \colon A \to B$ and $\epsilon_1 \colon G \to M$ quasi-isomorphisms of complexes with $A$ and $G$ bounded complexes of projectives. Then there exists an $\ai$-algebra structure on $A$, an $\ai$-module structure over $A$ on $G$, a quasi-isomorphism $\phi \colon A \to B$ of $\ai$-algebras extending $\phi_1$ and a quasi-isomorphism $\epsilon \colon G \to \phi_* M$ of $\ai$-modules over $A$ extending $\epsilon_1$. 

Moreover, if the $\ai$-structures on $B$ and $M$ are strictly unital, and $A = Q \oplus \bar{A}$ as graded modules such that $\partial(Q) = 0$ and $(Q \to A \to B) = \eta^B$, then the $\ai$-structures on $A$ and $G$ and the morphisms $\phi$ and $\epsilon$ can be chosen to be strictly unital. \qed
\end{corollary}

\begin{remark} \label{sign_convention}
There are various sign conventions for the Stasheff identities and for the relations to the shifted maps. We try to give an overview on the most common conventions and how they are related.

For the Stasheff identities of $\ai$-algebras there are mainly two different conventions: First there is the convention used in this article, which also, among others, is used in \cite{Stasheff:1963b,Merkulov:1999,Keller:2001,Loday/Vallette:2012}. For the other convention one replaces the sign $(-1)^{u+vw}$ in \cref{alg_stasheff} by $(-1)^{uv+w}$; this is used in \cite{LefevreHasegawa:2003,Sagave:2010}. One sees these conventions are equivalent by replacing $\mu^A_k$ by $(-1)^{\frac{k(k-1)}{2}} \mu^A_k$. The Stasheff identities for morphisms and homotopies need to be adjusted appropriately.

For modules the situation is a bit more complex. First the sign is chosen parallel to that of algebras described above. A second difference is in the connection to the bar construction. In this article we use the equivalence between $\ai$-module structures and square-zero coderivations on $\tcoa{\susp A} \otimes M$. Additionally, these structures are equivalent to square-zero coderivations on $\tcoa{\susp A} \otimes \susp M$. One obtains different Stasheff identities for morphisms of $\ai$-modules, whether one translates a morphism from $\tcoa{\susp A} \otimes M$ or $\tcoa{\susp A} \otimes \susp M$ to a morphism of $\ai$-algebras; see for example \cite[4.2]{Keller:2001} for the latter.

Some sources avoid the problem of a sign convention, by defining the $\ai$-structures as the shifted structures. This has the advantage of simplifying many formulas. The drawback is that it does not behave as an algebra anymore. 

There seem to be two `good' conventions for the relationship between $\ai$-structures and the corresponding shifted maps. There is the one used in this article, and then the one obtained by adding a factor $(-1)^{k-1}$ to the definition of the shifted map for every structure. In practice this choice does not have a big impact, but differences can for example be observed in the strictly unital condition when expressed using the shifted maps.
\end{remark}

\bibliographystyle{amsalpha}
\bibliography{refs}

\providecommand{\bysame}{\leavevmode\hbox to3em{\hrulefill}\thinspace}
\providecommand{\MR}{\relax\ifhmode\unskip\space\fi MR }
\providecommand{\MRhref}[2]{%
  \href{http://www.ams.org/mathscinet-getitem?mr=#1}{#2}
}
\providecommand{\href}[2]{#2}
\begin{thebibliography}{LPWZ04}

\bibitem[BCLP25]{Briggs/Cameron/Letz/Pollitz:2025}
Benjamin Briggs, James~C. Cameron, Janina~C. Letz, and Josh Pollitz,
  \emph{Koszul homomorphisms and universal resolutions in local algebra}, Forum
  Math. Sigma \textbf{13} (2025), e63.

\bibitem[Bur]{Burke:notes}
Jesse Burke, \emph{Koszul duality for represenations of an a-infinty algebra
  defined over a commutative ring},
  https://maths-people.anu.edu.au/~burkej/koszul/old-Koszul.pdf, (accessed July
  6, 2022), pp.~1--73.

\bibitem[Bur18]{Burke:2018}
\bysame, \emph{Transfer of {A}-infinity structures to projective resolutions},
  arXiv e-prints, 2018, arXiv:1801.08933v1, pp.~1--20.

\bibitem[CG08]{Cheng/Getzler:2008}
Xue~Zhi Cheng and Ezra Getzler, \emph{Transferring homotopy commutative
  algebraic structures}, J. Pure Appl. Algebra \textbf{212} (2008), no.~11,
  2535--2542. \MR{2440265}

\bibitem[DS95]{Dwyer/Spalinski:1995}
William~G. Dwyer and Jan Spali{\'n}ski, \emph{Homotopy theories and model
  categories}, Handbook of algebraic topology (Ioan~M. James, ed.),
  North-Holland, Amsterdam, 1995, pp.~73--126. \MR{1361887}

\bibitem[GL89]{Gugenheim/Lambe:1989}
Victor K. A.~M. Gugenheim and Larry~A. Lambe, \emph{Perturbation theory in
  differential homological algebra. {I}}, Illinois J. Math. \textbf{33} (1989),
  no.~4, 566--582. \MR{1007895}

\bibitem[GLS91]{Gugenheim/Lambe/Stasheff:1991}
Victor K. A.~M. Gugenheim, Larry~A. Lambe, and James~D. Stasheff,
  \emph{Perturbation theory in differential homological algebra. {II}},
  Illinois J. Math. \textbf{35} (1991), no.~3, 357--73. \MR{1103672}

\bibitem[GS86]{Gugenheim/Stasheff:1986}
Victor K. A.~M. Gugenheim and James~D. Stasheff, \emph{On perturbations and
  {$A_\infty$}-structures}, Bull. Soc. Math. Belg. S{\'e}r. A \textbf{38}
  (1986), 237--246. \MR{885535}

\bibitem[Gug82]{Gugenheim:1982}
Victor K. A.~M. Gugenheim, \emph{On a perturbation theory for the homology of
  the loop-space}, J. Pure Appl. Algebra \textbf{25} (1982), no.~2, 197--205.
  \MR{662761}

\bibitem[HK91]{Huebschmann/Kadeishvili:1991}
Johannes Huebschmann and Tornike Kadeishvili, \emph{Small models for chain
  algebras}, Math. Z. \textbf{207} (1991), no.~2, 245--280. \MR{1109665}

\bibitem[JL01]{Johansson/Lambe:2001}
Leif Johansson and Larry Lambe, \emph{Transferring algebra structures up to
  homology equivalence}, Math. Scand. \textbf{89} (2001), no.~2, 181--200.
  \MR{1868174}

\bibitem[Kad80]{Kadeishvili:1980}
Tornike~V. Kadeishvili, \emph{On the homology theory of fibre spaces}, Uspekhi
  Mat. Nauk \textbf{35} (1980), no.~3, 183--188. \MR{580645}

\bibitem[Kel01]{Keller:2001}
Bernhard Keller, \emph{Introduction to {$A$}-infinity algebras and modules},
  Homology Homotopy Appl. \textbf{3} (2001), no.~1, 1--35. \MR{1854636}

\bibitem[LH03]{LefevreHasegawa:2003}
Kenji Lef{\`e}vre-Hasegawa, \emph{Sur les {$A_\infty$}-catégories}, Ph.D.
  thesis, Universit{\'e} Paris 7 - Denis Diderot, 2003, p.~230.

\bibitem[LPWZ04]{Lu/Palmieri/Wu/Zhang:2004}
Di~Ming Lu, John~H. Palmieri, Quan~Shui Wu, and James~J. Zhang,
  \emph{{$A_\infty$}-algebras for ring theorists}, Algebra Colloq. \textbf{11}
  (2004), no.~1, 91--128. \MR{2058967}

\bibitem[LV12]{Loday/Vallette:2012}
Jean-Louis Loday and Bruno Vallette, \emph{Algebraic operads}, Grundlehren
  Math. Wiss., vol. 346, Springer, Heidelberg, 2012. \MR{2954392}

\bibitem[Mar06]{Markl:2006}
Martin Markl, \emph{Transferring {$A_\infty$} (strongly homotopy associative)
  structures}, Rend. Circ. Mat. Palermo (2) Suppl. (2006), no.~79, 139--151.
  \MR{2287133}

\bibitem[Mer99]{Merkulov:1999}
Sergei~A. Merkulov, \emph{Strong homotopy algebras of a {K}{\"a}hler manifold},
  Internat. Math. Res. Notices (1999), no.~3, 153--164. \MR{1672242}

\bibitem[Pet20]{Petersen:2020}
Dan Petersen, \emph{A closer look at {K}adeishvili's theorem}, High. Struct.
  \textbf{4} (2020), no.~2, 211--221. \MR{4133168}

\bibitem[Pro11]{Proute:2011}
Alain Prout{\'e}, \emph{{$A_\infty$}-structures. {M}od{\`e}les minimaux de
  {B}aues-{L}emaire et {K}adeishvili et homologie des fibrations}, Repr. Theory
  Appl. Categ. (2011), no.~21, 1--99, Reprint of the 1986 original, With a
  preface to the reprint by Jean-Louis Loday. \MR{2844537}

\bibitem[Sag10]{Sagave:2010}
Steffen Sagave, \emph{{DG}-algebras and derived {$A_\infty$}-algebras}, J.
  Reine Angew. Math. \textbf{639} (2010), 73--105. \MR{2608191}

\bibitem[Sta63a]{Stasheff:1963a}
James~Dillon Stasheff, \emph{Homotopy associativity of {$H$}-spaces. {I}},
  Trans. Amer. Math. Soc. \textbf{108} (1963), 275--292. \MR{0158400}

\bibitem[Sta63b]{Stasheff:1963b}
\bysame, \emph{Homotopy associativity of {$H$}-spaces. {II}}, Trans. Amer.
  Math. Soc. \textbf{108} (1963), 293--312. \MR{0158400}

\bibitem[Wei94]{Weibel:1994}
Charles~A. Weibel, \emph{An introduction to homological algebra}, Cambridge
  Stud. Adv. Math., vol.~38, Cambridge Univ. Press, Cambridge, 1994.
  \MR{1269324}

\end{thebibliography}

\end{document}